\documentclass[11pt]{amsart}
\usepackage[latin1]{inputenc}
\usepackage[english]{babel}
\usepackage[T1]{fontenc}
\usepackage{amsmath,amssymb}
\usepackage{stmaryrd}
\usepackage{enumerate}
\usepackage{amsthm}
\usepackage{mathrsfs}
\usepackage{geometry}
\usepackage{mathtools}
\usepackage{graphicx}
\usepackage{array}
\usepackage[draft,english]{fixme}
\usepackage[all,line]{xy}
\usepackage[usenames,dvipsnames,svgnames,table]{xcolor}
\usepackage{xfrac}
\mathtoolsset{showonlyrefs}
\newtheorem{thm}{Theorem}[section]
\newtheorem{prop}[thm]{Proposition}
\newtheorem{lemma}[thm]{Lemma}
\theoremstyle{definition}
\newtheorem{defin}[thm]{Definition}
\theoremstyle{remark}
\newtheorem{rem}[thm]{Remark}

\newtheorem{cor}[thm]{Corollary}

\newcommand{\iso}{\simeq}

\newcommand{\desc}{{\mathfrak{Desc}}}

\newcommand{\C}{\mathbb{C}}
\newcommand{\ccal}{\mathcal{C}}
\newcommand{\R}{\mathbb{R}}

\newcommand{\She}{\mathcal{O}}
\newcommand{\Ob}{X}
\newcommand{\F}{\mathcal{F}}
\newcommand{\QCoh}{\operatorname{QCoh}}

\newcommand{\Hom}{\operatorname{Hom}}

\newcommand{\pr}{{\operatorname{pr}}}

\newcommand{\id}{\operatorname{Id}}

\begin{document}

\title{Quasi-coherent Hecke category and Demazure Descent}
\author{Sergey Arkhipov}
\author{Tina Kanstrup}
\address{S.A. Matematisk Institut, Aarhus Universitet, Ny Munkegade, DK-8000 , Århus C, Denmark, email: hippie@qgm.au.dk} 
\address{T.K. Centre for Quantum Geometry of Moduli Spaces, Aarhus Universitet, Ny Munkegade, DK-8000 , Århus C, Denmark, email: tina@qgm.au.dk}
\dedicatory{To Boris Feigin on the occasion of his 60-th birthday\\ with gratitude and admiration}
\maketitle

\begin{abstract}
Let $G$ be a reductive algebraic group with a Borel subgroup $B$. We define the quasi-coherent Hecke category for the pair $(G,B)$. For any regular Noetherian $G$-scheme $X$ we construct a monoidal action of the Hecke category on the derived category of $B$-equivariant quasi-coherent sheaves on $X$. Using the action we define the Demazure Descent Data on the latter category and prove that the Descent category is equivalent to the derived category of $G$-equivariant sheaves on $X$.
\end{abstract}

\section{Introduction}
The present paper is the second one in the series devoted to the study of Demazure Descent. In \cite{AK} we introduced the notion of Demazure Descent Data on a triangulated category. A category with Demazure Descent Data is a higher analog for a representation of the degenerate Hecke algebra (see \cite{HLS}). 

Such representations arise naturally from geometry. Let $X$ be a compact real manifold 
acted on by a compact simple Lie group $G$ with a fixed maximal torus $T$. 
Harada et al constructed a natural action of the degenerate Hecke algebra of the corresponding type on the $T$-equivariant K-groups of $X$. They showed that the $G$-equivariant K-groups are identified naturally with the invariants for the action. In a way we categorify the construction from \cite{HLS} in our paper.

Let $G$ be a reductive algebraic group with a fixed Borel subgroup $B$. Recall that the (finite) Hecke algebra is defined classically as the algebra of $B(\mathbb{F}_q)$-biequivariant functions on the group $G(\mathbb{F}_q)$ with values in $\C$. The multiplication is provided by convolution. It turns out that the stack $B \backslash G/B$ is a universal geometric tool to produce "algebras", both in the usual and in the categorical sense. In particular, the categories of constructible sheaves and D-modules on $B \backslash G/B$ were studied in \cite{Tan}. They were used as a natural source of finite Hecke algebra actions on categories of geometric origin.

In the present paper, we consider the quasi-coherent Hecke category QCHecke$(G,B)$, the derived category of $B$-biequivariant quasi-coherent sheaves on $G$. For technical reasons we prefer to work with the equivalent category - the derived category of $G$-equivariant quasi-coherent sheaves on $G/B \times G/B$.

Let us outline the structure of the paper. We recall the standard definitions and facts related to equivariant quasi-coherent sheaves on a scheme in Section 2.
We introduce the monoidal category $\text{QCHecke}(G,B)$ in Section 3. In Section 4 we recall the definitions of Demazure Descent Data on a triangulated category and of the corresponding Descent category. 

Let $X$ be a regular Noetherian scheme acted on by a reductive algebraic group $G$. Section 5 is devoted to the construction of Demazure Descent Data on the derived category of $B$-equivariant quasicoherent sheaves on $X$ in terms of the natural monoidal action of  $\text{QCHecke}(G,B)$ on the category  by convolution.
\vskip 2mm
\noindent
{\bf Theorem:}
For $w \in W$ let $X_w$ denote the closure of the corresponding $G$ orbit in $G/B \times G/B$. Then the functors of convolution with the structure sheaves $\She_{X_w}$ define Demazure Descent Data on $D(\QCoh^B(X))$.
\vskip 2mm

In the last Section, we study the corresponding Descent category. We prove the central result of the paper: 
\vskip 2mm
\noindent
{\bf Theorem:}
$\desc(D(\QCoh^B(X)),D_w,w\in W)$ is equivalent to $D(\QCoh^G(X))$.

\subsection{Acknowledgements}
The authors are grateful to H.H. Andersen, C. Dodd, V. Ginzburg, M. Harada and R. Rouquier for many stimulating discussions. The project started in the summer of 2012 when the first named author visited IHES. S.A. is grateful to IHES for perfect working conditions. Part of the work was done while the second author visited R. Bezrukavnikov at MIT in the Fall 2013. T.K. would like to express her deepest gratitude to R. Bezrukavnikov for all that he taught her during her stay and for useful comments on previous versions of the text. T.K. would also like to thank MIT for perfect working conditions.

Both authors' research was supported in part by center of excellence grants "Centre for Quantum Geometry of Moduli Spaces" and by FNU grant "Algebraic Groups and Applications".

\subsection{Conventions.} In the present paper, we work over an algebraically closed field $k$ of characteristic zero. An algebraic group means an affine algebraic group scheme over $k$. All schemes are supposed to be Noetherian, of finite Krull dimension over $k$.

\section{Equivariant quasi-coherent sheaves on a scheme}

Below we collect the main facts about equivariant quasi-coherent sheaves to be used later.
In this Section, $K$ denotes a not necessarily reductive  algebraic group. 

Let $X$ be a $K$-scheme. Denote  the action (resp., the projection) map  $K \times X \to X$ by $\text{ac}$ (resp., by  $p$). Recall that a $K$-equivariance structure on a quasi-coherent sheaf $M$ on $X$ is given by an isomorphism $\theta$ between $p^*(M)$ and $\text{ac}^*(M)$ such that the further pull-backs of $\theta$ to $K \times K\times X$ satisfy the standard cocycle condition (see e.g. \cite{Bri}, Section 2, for the precise formulation).

The category  of quasi-coherent sheaves on $X$  (resp., of $K$-equivariant quasi-coherent sheaves on $X$) is denoted by $\QCoh(X)$ (resp., by $\QCoh^K(X)$). We have the forgetful functor 
$Oblv:\ \QCoh^K(X)\to \QCoh(X)$.

Below we always assume that $X$ is good enough, and any $K$-equivariant quasi-coherent sheaf has a uniformly bounded resolution by $K$-equivariant quasi-coherent sheaves locally free on $X$.

Let $f: X \to Y$ be a $K$-equivariant map of $K$-schemes. The functors of push-forward and pull-back are extended naturally to the categories of equivariant sheaves:
\begin{gather}
f^* : \QCoh^K(Y) \to \QCoh^K(X), \qquad (M,\theta) \mapsto (f^*M, f^* \theta \circ \text{canonical isomorphisms}),\\
f_* : \QCoh^K(X) \to \QCoh^K(Y), \qquad (M,\theta) \mapsto (f_*M,(\id \times f)_* \theta \circ \text{canonical isomorphisms}).
\end{gather}
Notice that both $f^*$ and $f_*$ commute with $Oblv$.

Let $K,H$ be  algebraic groups acting on a scheme $X$ so that the actions commute. Assume that $X$ admits an $H$-equivariant quotient  $q: X \to X/K$ which is a locally trivial principal $K$-bundle. Denote the quotient scheme $X/K$ by $Y$.
\begin{lemma}
The inverse image functor provides   an equivalence of  Abelian categories $ \QCoh^H(Y) \to \QCoh^{H \times K}(X)$.
\end{lemma}
\begin{proof}
See \cite{Bri}, discussion in Section 2.
\end{proof}

Let $X$ be a $K$-scheme. 
It is known that the forgetful functor $\QCoh^K(X)\to \QCoh(X)$ has an exact right adjoint functor denoted by $Av^K$, and for any $M\in \QCoh(X)$ the natural map $M\to Oblv\circ Av^K(M)$ is an embedding. It follows that the category $\QCoh^K(X)$ has enough injective objects  (see \cite{Bez}, Section 2).  Moreover, Varagnolo and  Vasserot state that under very mild restrictions on $X$, any unbounded complex in  $\QCoh^K(X)$ has a K-injective resolution and a K-flat resolution (see \cite{VV}, 1.5.6 and further paragraphs).

To avoid further restrictions on $X$, we work in the unbounded derived category $D\QCoh^K(X)$. The functor $Oblv$ is exact and extends to the functor $D\QCoh^K(X)\to D\QCoh(X)$. 

Let $f : X \to Y$ be a $K$-morphism of  normal Noetherian quasi-projective $K$-schemes.  
Consider the derived functors 
$$Lf^*:\  D\QCoh^K(Y)\to D\QCoh^K(X) \text{ and }Rf_*:\ D\QCoh^K(X)\to D\QCoh^K(Y).$$
 It is known that the functors  $Rf_*$ and $Oblv$ (resp., $Lf^*$ and $Oblv$) commute (see e.g. \cite{VV}, Lemma 1.5.9 and the discussion immediately after it). 

 Derived tensor products  in  $D\QCoh^K(Y)$ and  $D\QCoh^K(X)$ are denoted by $\overset{L}{\otimes}_Y$ and  $\overset{L}{\otimes}_X$ respectively. We have the projection formula as follows:

\begin{prop} 
For $N\in D\QCoh^K(Y)$ and $M\in D\QCoh^K(X)$ we have a canonical isomorphism  $Rf_*N \overset{L}\otimes_{Y}M \iso Rf_*(N \overset{L}\otimes_{X} Lf^*M)$.
\end{prop} 
\begin{proof}
See \cite{VV} section 1.5.8.
\end{proof}
Recall the equivariant  flat base change theorem. Let $g :Z \to Y$ be a flat $K$-morphism, where $Z$ is a normal quasi-projective $K$-scheme. Consider the Cartesian squre
\[ \xymatrix{Z \times_Y X \ar[d]^{f'} \ar[r]^{g'} & X \ar[d]^f \\ Z \ar[r]^g & Y} \]

\begin{prop} \label{equiprop} 
The standard adjunction map provides an isomorphism of functors $Lg^* Rf_* \iso R f'_* Lg'^*$.
\end{prop}

\section{Convolution and the quasi-coherent Hecke category}
\subsection{Convolution monoidal structure.} \label{monoidal}
Let $Z,Y$ and $X$ be  $K$-schemes. Consider the projections
\[ \xymatrix{& Z \times Y \times X \ar[dl]_{\pr_{12}} \ar[d]^{\pr_{13}} \ar[dr]^{\pr_{23}} & \\ Z \times Y & Z \times X & Y \times X} \]
The group $K$ acts on each of the four schemes in the diagram diagonally, and the projections are $K$-equivariant.

The convolution product $*$ is defined as follows:
\begin{gather} 
*: D(\QCoh^{K}(Z \times Y)) \times D(\QCoh^{K}(Y \times X))\to D(\QCoh^K(Z \times X)),\\
M_1* M_2 := R\pr_{13*} (L\pr_{12}^*M_1\overset{L} \otimes_{Z\times Y\times X} L\pr_{23}^*M_2).
\end{gather}
Suppose that $Z=Y=X$ is regular. In this case,  the convolution product becomes a monoidal structure
$$*: \ D(\QCoh^{K}(X\times X)) \times D(\QCoh^{K}(X \times X))\to D(\QCoh^K(X \times X))$$
in a weak sense: the associativity constraint $(M_1*M_2)*M_3\tilde{\to} M_1*(M_2*M_3)$  is not specified. 

In the same way, if $Z=Y$ is regular, the convolution product induces a monoidal action (in the same weak sense) of the monoidal category $ D(\QCoh^K(Y \times Y))$ on the category 
$ D(\QCoh^K(Y \times X))$.

\subsection{Comparing two convolutions.} The key technical statement in the proof of associativity of the monoidal structure above as well as of several more specific Lemmas below requires the following setup.

Suppose we have the $K$-schemes $X_1$, $X_2$ and $X_3$, $Z_{12}$, $Z'_{12}$ and $Z_{23}$. We are given the $K$-equivariant  flat maps
$$
p_1:\ Z_{12}\to X_1,\  p_2:\ Z_{12}\to X_2,\ q_2:\ Z_{23}\to X_2,\ q_3:\ Z_{23}\to X_3,
$$
and $\alpha:\ Z_{12}\to Z'_{12}\to Z_{12}$ such that $p'_1=p_1\circ \alpha$ and $p'_2=p_2\circ \alpha$ are also flat. 

Consider the standard projections
\begin{gather*}
\text{pr}_{12}:\ Z_{12}\times_{X_2} Z_{23}\to Z_{12},\ \text{pr}_{23}:\ Z_{12}\times_{X_2} Z_{23}\to Z_{23},\ \text{pr}_{13}:\ Z_{12}\times_{X_2} Z_{23}\to X_1\times X_3,\\
\text{pr}'_{12}:\ Z'_{12}\times_{X_2} Z_{23}\to Z'_{12},\ \text{pr}'_{23}:\ Z'_{12}\times_{X_2} Z_{23}\to Z'_{23},\ \text{pr}'_{13}:\ Z'_{12}\times_{X_2} Z_{23}\to X_1\times X_3.
\end{gather*}
We introduce the convolution products
\begin{gather*}
*:\ D(\QCoh^{K}(Z' _{12})) \times D(\QCoh^{K}(Z_{23}))\to D(\QCoh^K(X_1\times X_3)),\\
M_1* M_2 := R\pr_{13*} (L\pr_{12}^*(R\alpha_*(M_1))\overset{L} \otimes_{Z_{12}\times_{X_2} Z_{23}} L\pr_{23}^*(M_2))
\end{gather*}
and
\begin{gather*}
*':\ D(\QCoh^{K}(Z' _{12})) \times D(\QCoh^{K}(Z_{23}))\to D(\QCoh^K(X_1\times X_3)),\\
M_1*' M_2 := R\pr'_{13*} (L(\pr'_{12})^*(M_1)\overset{L} \otimes_{Z'_{12}\times_{X_2} Z_{23}} L(\pr'_{23})^*(M_2)).
\end{gather*}
\begin{lemma} \label{technical}
The convolutions $*$ and $*'$ are canonically isomorphic.
\end{lemma}
\begin{proof}
Denote the base change of the map $\alpha$ via $\pr_{12}$ by $\beta: Z'_{12}\times_{X_2} Z_{23}\to Z_{12}\times_{X_2} Z_{23}$.
The argument in the proof is standard and combines the base change via $\pr_{12}$ for the $R\alpha_*$ and the projection formula for the map $\beta$.
\end{proof}
\begin{rem}\label{technical2}
A typical special case in which Lemma \ref{technical} is applied is as follows. Take $X_1=X_2=X_3=X$. For a flat surjective  $K$-equivariant map $X\to Y$ consider
$$Z'_{12}=X\times_YX,\ Z_{12}=Z_{23}=Z_{13}=X\times X.
$$
Lemma \ref{technical} implies that convolution operations defined via $X\times_YX\times X$ and via $X\times X\times X$ coincide.
\end{rem}
\begin{rem}\label{unit}
In particular the unit object in the monoidal category $D(\QCoh^K(X\times X))$  is given by the structure sheaf of the diagonal in $X\times X$ denoted by $\She_{X_\Delta}$.
\end{rem}

\subsection{Convolution and correspondences.} \label{correspondences}
Let $X, Y_1, \ldots Y_n$ be  regular $K$-schemes. Suppose we are given flat surjective maps $\phi_i:\ X\to Y_i$, $i=1,\ldots,n$. Denote the fiber product $X\times_{Y_i}X\subset X\times X$ by 
$\alpha_i:\ X_i\to X\times X$. Consider the iterated fibered product 
$$
Z_{i_1,\ldots,i_k}:=X_{i_1}\times_X\ldots\times_XX_{i_k}=X\times_{Y_{i_1}}X\ldots\times_{Y_{i_k}}X\subset X^{k+1}.
$$
We have the map provided by the projections to the first and last factors
$$
\alpha_{i_1,\ldots,i_k}:\ Z_{i_1,\ldots,i_k}\to X\times X.
$$
Denote the image of the map by $X_{i_1,\ldots,i_k}\subset X\times X$. All the defined schemes are acted naturally by $K$ and all the defined maps are $K$-equivariant.

Consider the sheaves $M_i:=\alpha_{i\ *}({\mathcal{O}_{X_i}})$. 
The category $D(\QCoh^{K}(X \times X))$ acts on the category $D(\QCoh^{K}(X))$ by convolution. Denote the functor of convolution with $M_i$ by $D_i$.

\begin{lemma}\label{comonad}
The functor $D_i:\ D(\QCoh^{K}(X))\to D(\QCoh^{K}(X))$ is isomorphic to the functor $R\phi_i^*\circ L \phi_{i*}$.
\end{lemma}
\begin{proof}
Denote the two projections $X_i\to X$ by $q_{1,i}$ and $q_{2,i}$.
Let $N\in D(\QCoh^{K}(X))$. Apply Lemma \ref{technical} as suggested in Remark \ref{technical2}.
Notice that $M_i*N\widetilde{\to}Rq_{1,i*}Lq_{2,i}^*(N)$. Applying flat base change we obtain the statement of the Lemma.
\end{proof}
\begin{cor} Each functor $D_i$ is isomorphic to a comonad. 
Suppose additionally that the maps $\phi_i:\ X\to Y_i$ are rational. Then the comonads $D_i$ are coprojectors, i.e. the coproduct maps $D_i\to D_i\circ D_i$ are isomorphisms of functors.
\end{cor}
Our goal is to describe the composition of the functors $D_{i_1}\circ\ldots\circ D_{i_n}$ explicitly. Denote $R\alpha_{i_1,\ldots,i_k*}(\mathcal{O}_{ Z_{i_1,\ldots,i_k}})\in D(\QCoh^{K}(X \times X))$
by $M_{i_1,\ldots,i_k}$.

\begin{lemma}\label{composition}

We have a natural isomorphism of objects in $D(\QCoh^{K}(X \times X))$
$$
M_{i_1}*\ldots*M_{i_k}\widetilde{\to} M_{i_1,\ldots,i_k}.
$$
\end{lemma}
\begin{proof} We proceed by induction. Like in the proof of the previous Lemma, consider the two projections $q_{1,i}$ and $q_{2,i}: X\times_{Y_i}X\times X\to X\times X$ and notice that
$$ 
M_{i_1}*\ldots*M_{i_k}\widetilde{\to}M_{i_1}*M_{i_2,\ldots,i_k}\widetilde{\to}Rq_{1,i*}Lq_{2,i}^*(M_{i_2,\ldots,i_k}).
$$
Recall that $Z_{i_1,\ldots,i_k}=X\times_{Y_i}Z_{i_2,\ldots,i_k}$. Applying base change, we get that $M_{i_1}*\ldots*M_{i_k}$ is isomorphic to the direct image of $\mathcal{O}_{ Z_{i_1,\ldots,i_k}}$
for the composition of maps
$$
Z_{i_1,\ldots,i_k}=X_{i_1}\times_X Z_{i_2,\ldots,i_k}\to X_{i_1}\times X\to X\times X\times X\overset{\text{pr}_{13}}{\to}X\times X.
$$
The latter map coincides with $\alpha_{i_1,\ldots,i_k}$.
\end{proof}
\begin{cor} \label{rational} Suppose additionally that the map $\alpha_{i_1,\ldots,i_k}:\ Z_{i_1,\ldots,i_k}\to X_{i_1,\ldots,i_k}$ is rational. Then the convolution product $M_{i_1}*\ldots*M_{i_k}$ is isomorphic to the structure sheaf of $X_{i_1,\ldots,i_k}$.
\end{cor}

\subsection{Quasi-coherent Hecke category.}
Fix a reductive algebraic group $K$ with an algebraic subgroup $H$. Consider the $K$-scheme $Y=K/H$.
 \begin{defin}
The monoidal category
\[ (D(\QCoh^K(K/H \times K/H)), *) \]
is called the quasi-coherent Hecke category and it is denoted by QCHecke$(K,H)$. 
\end{defin}

Notice that for a $K$-scheme $X$ we have
\[ D(\QCoh^H(X)) \iso D(\QCoh^K(K/H \times X)). \]
Taking $Z=Y=K/H$ in the setting \ref{monoidal} we get the monoidal action 
\[ \text{QCHecke}(K,H) \times D(\QCoh^H(X)) \to D(\QCoh^H(X)). \]

\section{Demazure Descent}

\subsection{Notations}

From now on $G$ is a reductive algebraic group. Let $T$ be a Cartan subgroup of $G$ and let $X$ (resp., $Y$) be the weight (resp., the coroot) lattice of $G$. Choose a Borel subgroup $T\subset B\subset G$. 

Denote the set of roots for $G$ by $\Phi=\Phi^+\sqcup \Phi^-$. Let $\{\alpha_1, \dots, \alpha_n\}$ be the set of simple roots. The Weyl group $W=\text{Norm}(T)/T$ of the fixed maximal torus acts naturally on the lattices $X$ and $Y$ and on the $\R$-vector spaces spanned by them, by reflections in root hyperplanes. The simple reflection corresponding to a simple root  $\alpha_i$ is denoted by $s_i$. 

For an element $w\in W$ denote the length of a minimal expression of $w$ via the generators $\{s_i\}$ by $\ell(w)$. The unique longest element in $W$ is denoted by $w_0$. We have a partial ordering on $W$ called the Bruhat ordering: $w' \leq w$ if there exists a reduced expression for $w'$ that can be obtained from a reduced expression for $w$ by deleting a number of simple reflections. 

The monoid $\text{Br}^+$ with generators $\{T_w,\ w\in W\}$ and relations 
$$T_{w_1}T_{w_2}=T_{w_1w_2}  \text{ if } \ell(w_1)+\ell(w_2)=\ell(w_1w_2) \text{ in } W
$$
is called the braid monoid of $G$.

\subsection{Demazure Descent.}

Fix a reductive algebric group $G$ with the Weyl group $W$ and the braid monoid $\text{Br}^+$. 

\begin{defin}
A weak braid monoid action on a triangulated category $\ccal$ is a collection of triangulated functors
\[ D_w : \ccal \to \ccal, \qquad w \in W \]
satisfying  braid monoid relations, i.e. for all $w_1, w_2 \in W$ there exist isomorphisms of functors  \[ D_{w_1} \circ D_{w_2} \iso D_{w_1w_2}, \qquad \text{if } \ell(w_1 w_2)= \ell(w_1)+\ell(w_2). \]
\end{defin}
Notice that we neither fix the braid relations isomorphisms nor impose any additional relations on them. 

Recall that a comonad structure $(\epsilon,\eta)$ on an endofunctor $D: \ccal\to \ccal$ is a structure of a co-associative coalgebra on $D$ in the monoidal category of endofunctors. Here $\epsilon$ (resp., $\eta$) denotes the counit (resp., the coproduct) morphism for the comonad. A comonad $D$ is a co-projector if the coproduct map   $D \to D\circ D$ is an isomorphism.

\begin{defin}
Demazure Descent Data on a triangulated category $\ccal$ is a weak braid monoid action $\{D_w\}$ together with a co-projector structure $(\epsilon_k,\eta_k)$ on the functor $D_{s_k}$ for every simple reflection $s_k$. 
\end{defin}

\subsection{The descent category.}
Consider a triangulated category $\ccal$ with a fixed Demazure Descent Data $\{D_w,w\in W\}$.
\begin{defin}
The Descent category  $\desc(\ccal,D_w,w\in W)$ is the full subcategory in $\ccal$ with objects $M$ such that for all $k$ the cones of all counit maps 
$D_{s_k}(M) \stackrel{\epsilon_k}{\to} M$ are isomorphic to $0$. 
\end{defin}

\begin{rem}
Suppose that $\ccal$ has functorial cones. Then  $\desc(\ccal,D_w,w\in W)$ is a full triangulated subcategory in $\ccal$ since it is the intersection of kernels of  the  functors $\text{Cone}(D_{s_k} \to Id)$. However, one can prove this statement not using functoriality of cones. Namely, an object $M\in \ccal$ is an object in $\desc(\ccal,D_w,w\in W)$  if and only if all counit maps $D_{s_k} (M)\to M$ are isomorphisms, and it is easy to see that this condition is stable under taking shifts and cones.
\end{rem}

\begin{lemma}
An object $M \in \desc(\ccal,D_w,w\in W)$ is naturally a comodule over each $D_{s_k}$.
\end{lemma}
\begin{proof}
By definition the comonad maps
\[ \eta: D_{s_k} \to D_{s_k}^2, \qquad \epsilon: D_{s_k} \to \text{Id} \]
make the following diagrams commutative
\[ \xymatrix{& D_{s_k}(M) \ar@{=}[ld] \ar[d]_{\eta_M} \\ \text{Id} \circ D_{s_k}(M) & D_{s_k}^2(M) \ar[l]^(0.4){\epsilon_{D_{s_k}(M)}}} \qquad \xymatrix{D_{s_k}(M) \ar[r]^{\epsilon_M} & M \\ D_{s_k}^2(M) \ar[u]^{D_{s_k} (\epsilon_M)} \ar[r]_{\epsilon_{D_{s_k}(M)}} & D_{s_k}(M) \ar[u]_{\epsilon_M}}  \]
Notice that the rerequirement that 
Cone$(D_{s_k}(M) \to M) \iso 0$ is equivalent to invertibility of the counit map $\epsilon_M$. By functoriality $D_{s_k}(\epsilon_M)$ is an isomorphism with inverse $D_{s_k}(\epsilon_M ^{-1})$. By the first diagram  $\epsilon_{D_{s_k}M}$ is inverse to $\eta_M$. We get the commutative diagram
\[ \xymatrix{M \ar[r]^{\epsilon_M^{-1}} \ar[d]_{\epsilon_M^{-1}} & D_{s_k}(M) \ar[d]^{\eta_M} \\ D_{s_k}(M) \ar[r]_{D_{s_k}(\epsilon_M^{-1})} & D_{s_k}^2(M) }\]
The counit axiom is verified similarly. Thus, $M$ is a $D_{s_k}$-comodule with the coaction $\epsilon_M^{-1}$.
\end{proof}
\begin{rem}
Recall that in the usual Descent setting either in Algebraic Geometry or in abstract Category Theory (Barr-Beck Theorem) Descent Data includes a pair of adjoint functors and their composition which is a comonad. By definition, the Descent category for such data is the category of comodules over this comonad. Our definition of $\desc(\ccal,D_w,w\in W)$  for Demazure  Descent Data formally is not about comodules, yet the previous Lemma demonstrates that every object of $\desc(\ccal,D_w,w\in W)$ is naturally  equipped with structures of a comodule over each $D_k$ and any morphism in $\desc(\ccal,D_w,w\in W)$ is a morphism of $D_k$-comodules.
\end{rem}

\section{Demazure Descent for $D(\QCoh^G(G/B \times X))$}
Let $X$ be a scheme equipped with an action of a reductive algebraic group $G$.  For every element of the Weyl group $w \in W$ we construct  a functor $D_w$ acting on the category $D(\QCoh^G(G/B \times X))$. The functor is defined in terms of the monoidal action of $\operatorname{QCHecke}(G,B)$.

\subsection{Bott-Samelson varieties.}
Recall that the orbits for the  the diagonal action of $G$  on $G/B \times G/B$ are in bijection with elements of the Weyl group \cite[Theorem 3.1.9]{CG}. For $w \in W$ let $X_w$ be the closure of the corresponding orbit. The structure sheaves of the  orbit closures $\She_{X_w}, w\in W$,  ar objects of   the category $\operatorname{ QCHecke}(G,B)$.
 
 We will use an explicit description of the orbit closures. \label{orbits}
 For a simple root $\alpha_k$  denote the standard parabolic subgroup corresponding to $\alpha_k$ and containing $B$ by $P_k$.  We have   
$X_{s_k} =G/B\times_{G/P_k}G/B$.  To simplify 
the notations below we write $X_k$ for $X_{s_k}$. Notice that the subscheme $X_{k} \subset G/B \times G/B$ is regular and each of the projections to the factors $p_1,p_2:\ X_{s_k} \to G/B$ is flat.
Thus we are in the setting of \ref{correspondences}.

For an element $w\in W$, fix a reduced expression $w=s_{i_1}\ldots s_{i_k}$ where $k=\ell(w)$. 
The iterated fibered product 
$$
Z_{i_1,\ldots,i_k}:=X_{i_1}\times_{G/B}\ldots\times_{G/B}X_{i_k}=G/B\times_{G/P_{i_1}}G/B\ldots\times_{G/P_{i_k}}G/B\subset (G/B)^{k+1}
$$
is called the Bott-Samelson variety for $w$ (and for the fixed reduced expression of it).

We have the map provided by the projections to the first and last factors
$$
\alpha_{i_1,\ldots,i_k}:\ Z_{i_1,\ldots,i_k}\to G/B\times G/B.
$$
The main geometric result we use is the following  theorem due to Cline, Parshall and Scott.

\begin{thm} \cite{CPS}\label{bottsamelson}\

\begin{enumerate}
\item
The image of the map $\alpha_{i_1,\ldots,i_k}$ coincides with $X_w$.
\item
The map $\alpha_{i_1,\ldots,i_k}:\ Z_{i_1,\ldots,i_k}\to X_w$ is birational, in particular $$R\alpha_{i_1,\ldots,i_k*}(\mathcal{O}_{Z_{i_1,\ldots,i_k}})=\mathcal{O}_{X_w}.$$
\end{enumerate}
\end{thm}
\subsection{Demazure functors.}
Consider the functor
$$
D_w : D( \QCoh^G(G/B \times X)) \to D( \QCoh^G(G/B \times X)),\
D_w(M) :=\She_{X_w} * M.
$$
Below we prove that the functors $D_w, w\in W,$ form Demazure Descent Datum on the category $D( \QCoh^G(G/B \times X))$. 

\begin{rem}
The monoidal structure on the category $\operatorname{ QCHecke}(G,B)$ is defined in the weak sense, in particular, we can not talk about a co-associative coalgebra in this category. Thus the two parts of the definition of Demazure Descent Data are to be treated separately. Braid relations between the objects $\She_{X_w}, w\in W$ are checked in  $D( \QCoh^G(G/B \times X))$, but co-projector structures on $X_{\Ob_{s_i}}$ are defined using different considerations.
\end{rem}
Consider the projection $p_i:\ G/B \times X \to G/P_i\times X$ and the corresponding inverse and direct image functors 
$$
Lp_{i}^*,\ Rp_{i*}:\
D( \QCoh^G(G/B \times X))\overset{\longleftarrow}{\longrightarrow}
D( \QCoh^G(G/P_i \times X)).
$$
 The endofunctor $D_i $  of the category $D( \QCoh^G(G/B \times X))$ given by their composition carries a canonical strucutre of a comonad (see for example \cite[section VII.6]{Mac}). Since the map is rational, the projection formula implies that the comonad is co-projector: the coproduct map $D_i\to D_i\circ D_i$ is an isomorphism.
\begin{lemma}
The functors $D_{s_i}$  and $D_i$ are isomorphic.
\end{lemma}

\begin{proof}
The statement follows immediately from Lemma \ref{comonad}.
\end{proof}

Notice that the associativity for the monoidal action of $\operatorname{ QCHecke}(G,B)$ on the category $D( \QCoh^G(G/B \times X))$ implies that all relations up to a non-specified isomorphism can be checked in the Hecke category.

\begin{prop}\label{BraidRel}
Let $w=s_{k_1} \cdots s_{k_n}$ be a reduced expression. Then 
$\She_{X_{k_1}}* \ldots * \She_{X_{k_n}}$ and  $\She_{X_w}$  are isomorphic as objects in $\operatorname{ QCHecke}(G,B)$.
\end{prop}
\begin{proof}
By Lemma \ref{composition}, the convolution product $\She_{X_{k_1}}* \ldots * \She_{X_{k_n}}$ is isomorphic to the direct image of the structure sheaf of the Bott-Samelson variety $R\alpha_{i_1,\ldots,i_k*}(\mathcal{O}_{Z_{i_1,\ldots,i_k}})$. By Theorem \ref{bottsamelson}, the latter object is isomorphic to the structure sheaf of $X_w$.
\end{proof}

Now we are prepared to prove the central result of the paper.

\begin{thm} \label{DescentData}
The functors $\{D_w, w \in W\}$ form  Demazure Descent Data on the category $D(\QCoh^G(G/B \times X))$.
\end{thm}

\begin{proof}[Proof of Theorem \ref{DescentData}]
We have proved that each of the functors $D_i$ is a comonad and the coproduct maps are isomorphisms of functors. 

It remains  to show that for all $w_,w_2 \in W$ with $\ell(w_1 w_2)=\ell(w_1) +\ell(w_2)$ we have
\[ D_{w_1} \circ D_{w_2} \iso D_{w_1 w_2}. \]
Fix  reduced expressions for the Weyl gorup elements  $w_1=s_{k_1} \cdots s_{k_n}$ and $w_2=s_{j_1} \cdots s_{j_m}$. Since $\ell(w_1 w_2)=\ell(w_1) +\ell(w_2)$ the expression $w_1w_2=s_{k_1} \cdots s_{k_n} s_{j_1} \cdots s_{j_m}$. For any $M\in D(\QCoh^G(G/B \times X))$ we obtain
\begin{align}
D_{w_1} \circ D_{w_2} (M)& \widetilde{\to}\She_{X_{w_1}} * \She_{X_{w_2}} * M \\
& \widetilde{\to} \She_{X_{k_1}} * \cdots * \She_{X_{k_n}} * \She_{X_{j_1}} * \cdots * \She_{X_{j_m}} * M \\
& \widetilde{\to} \She_{X_{w_1 w_2}} * M\\
& \widetilde{\to} D_{w_1 w_2}(M).
\end{align}
\end{proof}

\section{Descent category}

Consider the Descent category for the consturcted Demazure Descent Data on the category $D(\QCoh^G(G/B \times X))$.

\begin{thm} \label{descent} The descent category 
$\desc(D(\QCoh^B(X)),D_w,w\in W)$ is equivalent to $D(\QCoh^G(X))$.
\end{thm}

Notice that the inverse image functors $Lp_k^*:\ D(\QCoh^G(G/P_i \times X))\to D(\QCoh^G(G/B \times X))$ are fully faithful by projection formula since the maps are rational. The same is true for 
$Lp^*:\ D(\QCoh^G(\times X))\to D(\QCoh^G(G/B \times X))$. Here $p$ denotes the projection $G/B\times X \to X$.

\begin{lemma} \label{D(M)=M}
An object $M$ in $D(\QCoh^G(G/B \times X))$ belongs to the essential image of $Lp_k^* $ if and only if the coaction map $ M\to D_k(M)$is an isomorphism.
\end{lemma}
\begin{proof}
We identified the functor $D_k$ with the composition $Lp_k^* R p_{k*}$. Thus, $M \iso D_k(M)$ implies that $M$ belongs to the essential image of  $Lp_k^*$. Assume that $M = Lp_k^*(N)$ for some $N \in D(\QCoh^{G}(G/P_k \times X))$. By projection formula, the adjunction map $\id \to Rp_{k*}p_k^*$ is an isomorphism. Thus we have
\[ D_k(M) \iso Lp_k^* R p_{k*} Lp_k^*(N) \iso Lp_k^* (N)=M . \qedhere \]
\end{proof}

\begin{rem}
The same argument shows that an object $M$ in $D(\QCoh^G(G/B \times X))$ belongs to the essential image of $Lp^*$ if and only if $D_{w_0}(M)$ is isomorphic to  $M$.
\end{rem}

\begin{proof}[Proof of Theorem \ref{descent}]
Let $M \in \desc(D(\QCoh^B(X)),D_w,w\in W)$. For every simple root $\alpha_k$ the object  $D_k(M)$  is isomorphic to $M$. Choose a reduced expression $s_{k_1} \cdots s_{k_n}$ for $w_0$. We have
\begin{align}
D_{k_1} \circ \cdots \circ D_{k_n} (M)\iso D_{w_0}(M).
\end{align}
It follows that $M$ belongs to  the essential image of  $D(\QCoh^G(X))$.  

In particular,  the descent category $\desc(D(\QCoh^B(X)),D_w,w\in W)$ is a full subcategory in the essential image of the functor $Lp^*$. 

To prove the other embedding, notice that 
the map $p$ factors as $$G/B \times X \to G/P_k \times X \to  X.$$ It follows that  the essential image of $Lp^*$ is a full subcategory in the essential image of $Lp_k^*$ for all $k$.  This completes the proof of the Theorem.
\end{proof}

\end{document}